\documentclass[a4paper,10pt]{amsart}
\usepackage{amsmath,amssymb}
\usepackage{amscd}
\newtheorem{thm}{Theorem}[section]
\newtheorem{prop}[thm]{Proposition}
\newtheorem{lemma}[thm]{Lemma}

\newtheorem{exam}[thm]{Example}
\theoremstyle{remark}
\newtheorem{remark}[thm]{Remark}
\newcommand{\id}{{\rm{id}}}
\newcommand{\Ad}{{\rm{Ad}}}
\newcommand{\Hom}{{\rm{Hom}}}

\newcommand{\BN}{\mathbf N}

\newcommand{\BC}{\mathbf C}

\newcommand{\BB}{\mathbf B}

\newcommand{\Aut}{{\rm{Aut}}}

\newcommand{\CalE}{{\mathcal{E}}}
\newtheorem{Def}{Definition}[section]

\title{Free coactions of a finite dimensional $C^*$-Hopf algebra and strong Morita equivalence}
\author{Kazunori Kodaka and Tamotsu Teruya}
\address{Department of Mathematical Sciences, Faculty of Science, Ryukyu
\endgraf
University, Nishihara-cho, Okinawa, 903-0213, Japan}
\address{Faculty of Education, Gunma University, 4-2 Aramaki-machi, Maebashi City,
\endgraf
Gunma, 371-8510, Japan}
\address{\sl{E-mail address}: \rm{kodaka@math.u-ryukyu.ac.jp}}
\address{\sl{E-mail address}: \rm{teruya@gunma-u.ac.jp}}

\keywords{conditional expectations, finite dimensional $C^*$-Hopf algebras, free coactions,
strong Morita equivalence}

\subjclass[2010]{46L05}

\begin{document}

\begin{abstract}
We shall introduce a notion of free coactions of a finite dimensional $C^*$-Hopf algebra on a $C^*$-algebra
modifying a notion of free actions of a discrete group on a $C^*$-algebra and we shall study several properties
on coactions of a finite dimensional $C^*$-Hopf algebra on $C^*$-algebras, which are relating to strong
Morita equivalence for inclusions of $C^*$-algebras.
\end{abstract}
\maketitle

\section{Introduction}\label{sec:intro} We shall introduce a notion of free coactions of a finite dimensional
$C^*$-Hopf algebra on a $C^*$-algebra modifying a notion of free actions of a discrete group on a
$C^*$-algebra, which is defined in Zarikian \cite {Zarikian:expectation} and we shall give a result similar
to \cite [Theorem 3.1.2]{Zarikian:expectation}. Also, we discuss the relations of the Rokhlin property,
the freeness, outerness and saturatedness of coactions of a finite dimensional $C^*$-Hopf algebra on
a $C^*$-algebra. Furthermore, we show that strong Morita equivalence for coactions preserves the
freeness of coactions of a finite dimensional $C^*$-Hopf algebra on a $C^*$-algebra using the result
similar to \cite [Theorem 3.1.2]{Zarikian:expectation}.
\par
For an algebra $A$, we denote by $1_A$ and $\id_A$ the unit element in $A$ and the identity map
on $A$, respectively. If no confusion arises, we denote them by $1$ and $\id$, respectively.
Also, we denote by $M(A)$ the multiplier $C^*$-algebra of $A$.
\par
Let $\pi$ be a homomorphism of $A$
to a $C^*$-algebra $B$ with $\overline{\pi(A)B}=B$. Then there is the unique strictly continuous
homomorphism of $M(A)$ to $M(B)$ extending $\pi$ to $M(A)$ by Jensen and Thomsen
\cite [Corollary 1.1.15]{JT:KK}. We denote it by $\underline{\pi}$.
\par
For each $n\in\BN$, we denote
by $M_n (\BC)$ the $n\times n$-matrix algebra over $\BC$ and $I_n$ denotes the unit element
in $M_n (\BC)$.
\par
Let $X$ be an $A-B$-equivalence bimodule. For any $a\in A$, $b\in B$, $x\in X$,
we denote by $a\cdot x$ the left $A$-action on $X$ and by $x\cdot b$ the right $B$-action on $X$,
respectively.
Let $\widetilde{X}$ be the dual $B-A$-equivalence bimodule of $X$ and $\widetilde{x}$ denotes
the element in $\widetilde{X}$ induced by an element $x\in X$. Also, we regard $X$ as a Hilbert
$M(A)-M(B)$-equivalence bimodule in the sense of Brown, Mingo and Shen \cite {BMS:quasi} as follows:
Let $\BB_B (X)$ be the $C^*$-algebra of all adjointable right $B$-linear operators on $X$.
We note that a right $B$-linear operator on $X$ is bounded. Then $\BB_B (X)$ can be identified with
$M(A)$. Similarly let ${}_{A} \BB (X)$ be the $C^*$-algebra of all adjointable left $A$-linear operators on
$X$ and ${}_A \BB (X)$ is identified with $M(B)$. In this way, we regard $X$ as a Hilbert
$M(A)-M(B)$-bimodule. Let $\Aut (X)$ be the group of all bijective linear maps on $X$.
 
\section{Preliminaries}\label{sec:pre} Let $H$ be a finite dimensional $C^*$-Hopf algebra. We denote
its comultiplication, counit and antipode by $\Delta$, $\epsilon$ and $S$, respectively.
We shall use Sweedler's notation $\Delta(h)=h_{(1)}\otimes h_{(2)}$ for any $h\in H$, which suppresses
a possible summation when we write comultipications. We denote by $N$ the dimension of $H$.
Let $H^0$ be the dual $C^*$-Hopf algebra of $H$. We denote its comultiplication, counit and antipode
by $\Delta^0$, $\epsilon^0$ and $S^0$, respectively. There is a distinguished projection $e$ in $H$.
We note that $e$ is the Haar trace on $H^0$. Also, there is a distinguished projection $\tau$ in $H^0$
which is the Haar trace on $H$. Since $H$ is finite dimensional, $H\cong\oplus_{k=1}^L M_{f_k}(\BC)$
and $H^0 \cong\oplus_{k=1}^K M_{d_k}(\BC)$ as $C^*$-algebras. Let
$\{v_{ij}^k \, | \, k=1,2,\dots, L, \, i, j=1, 2,\dots, f_k \}$ be a system of matrix units of $H$.
Let $\{w_{ij}^k \, | \, k=1, 2, \dots, K, \, i, j=1, 2,\dots, d_k \}$ be a basis of $H$ satisfying
Szyma\'nski and Peligrad's \cite [Theorem 2.2,2]{SP:saturated}, which is called a system of
\sl
comatrix units
\rm
of $H$, that is, the dual basis of a system of matrix units of $H^0$.
Also, let $\{\phi_{ij}^k \, | \, k=1, 2,\dots, K, \, i, j=1, 2, \dots, d_k \}$ and
$\{\omega_{ij}^k \, | \, k=1, 2, \dots, L, \, i, j=1, 2, \dots, f_k \}$ be systems of matrix units and comatrix
units of $H^0$, respectively. Let $A$ be a $C^*$-algebra.

\begin{Def}\label{pre1} By a coaction of $H^0$ on $A$ we mean a homomorphism $\rho$
of $A$ to $A\otimes H^0$ satisfying the following conditions:
\newline
(1) $\overline{\rho(A)(A\otimes H^0)}=A\otimes H^0$,
\newline
(2) $(\id\otimes\epsilon^0 )(\rho(a))=a$ for any $a\in A$,
\newline
(3) $(\rho\otimes\id)\circ\rho=(\id\otimes\Delta^0 )\circ\rho$.
\end{Def}

Let $\Hom(H, A)$ be the linear space of all linear maps from $H$ to $A$.
Since $H$ is finite dimensional, $\Hom(H, A)$ is isomorphic to $A\otimes H^0$.  We
identify $\Hom(H, A)$ with $A\otimes H^0$. For any element $x\in A\otimes H^0$, we denote
by $\widehat{x}$ the element in $\Hom(H, A)$ induced by $x$.
\par
For a coaction $\rho$ of $H^0$ on a $C^*$-algebra $A$, we consider the action of $H$ on $A$ defined
by
$$
h\cdot_{\rho}a=\widehat{\rho(a)}(h)=(\id\otimes h)(\rho(a))
$$
for any $a\in A$, $h\in H$. We call it the action of $H$ on $A$ induced by $\rho$. Let $A\rtimes_{\rho}H$
be the crossed product of $A$ by the action of $H$ on $A$ induced by $\rho$. Let $a\rtimes_{\rho}h$
be the element in $A\rtimes_{\rho}H$ induced by elements $a\in A$, $h\in H$. Let $\widehat{\rho}$ be the
dual coaction of $H$ on $A\rtimes_{\rho}H$ defined by
$$
\widehat{\rho}(a\rtimes_{\rho}h)=(a\rtimes_{\rho}h_{(1)})\otimes h_{(2)}
$$
for any $a\in A$, $h\in H$. Let $E_1^{\rho}$ be the canonical conditional expectation from $A\rtimes_{\rho}H$
onto $A$ defined by
$$
E_1^{\rho}(a\rtimes_{\rho}h)=\tau(h)a
$$ 
for any $a\in A$, $h\in H$.

\section{Coactions of a finite dimensional $C^*$-Hopf algebra}\label{sec:coaction}
Let $H$ and $H^0$ be as in Section \ref {sec:pre}. Let $\rho$ be a coaction of $H^0$ on
a $C^*$-algebra $A$.

\begin{Def}\label{def:coaction1} We say that $\rho$ is
\sl
free
\rm
if $\rho$ satisfies the following:
If $x\in M(A)\otimes H^0$ satisfies that
$$
x\rho(a)=(a\otimes 1^0 )x
$$
for any $a\in A$, then
$$
x\in (A' \cap M(A))\otimes\BC\tau .
$$
\end{Def}

We show that the above property is an extension of the freeness of discrete groups in
\cite {Zarikian:expectation} to coactions of finite dimensional $C^*$-Hopf algebras.
\par
Let $G$ be a finite group and $C(G)$ the finite dimensional $C^*$-Hopf algebra of all $\BC$-
valued functions on $G$. We denote by $1^0$ the unit element in $C(G)$.
Let $\alpha$ be an action of $G$ on a $C^*$-algebra $A$. Let $\rho_{\alpha}$
be the coaction of $C(G)$ on $A$ induced by $\alpha$, that is,
$$
\rho_{\alpha}(a)=\sum_{t\in G}\alpha_t (a)\delta_t
$$
for any $a\in A$, where $\delta_t$ is the function on $G$ defined by
$$
\delta_t (s)= \begin{cases} 1 & \text{if $s=t$} \\
0 & \text{if $s\ne t$} \end{cases}
$$
We recall the definition that $\alpha$ is a free action of $G$ on a $C^*$-algebra $A$.

\begin{Def}\label{def:coaction2} We say that $\alpha$ is
\sl
free
\rm
if $\alpha_t$ is a free automorphism of $A$ for any $t\in G\setminus\{\imath\}$, that is,
if an element $x\in M(A)$ satisfies that $xa=\alpha_t (a)x$ for any $a\in A$, then $x=0$, where
$\imath$ is the unit element in $G$.
\end{Def}

\begin{prop}\label{prop} Let $\alpha$ be an action of a finite group $G$ on a $C^*$-algebra $A$
and $\rho_{\alpha}$ the coaction of $C(G)$ on $A$ induced by $\alpha$. Then the following are
equivalent:
\newline
$(1)$ The action $\alpha$ is free,
\newline
$(2)$ The coaction $\rho_{\alpha}$ is free.
\end{prop}
\begin{proof} $(1)\Rightarrow(2)$: We suppose that $\alpha$ is free. We suppose that an element
$x\in M(A)\otimes C(G)$ satisfies that for any $a\in A$, $x\rho_{\alpha}(a)=(a\otimes 1^0 )x$. Then
$$
\sum_{t\in G}x(\alpha_t (a)\otimes\delta_t )=(a\otimes 1^0 )x .
$$
Since $x\in M(A)\otimes C(G)$, we can write that $x=\sum_{t\in G}x_t \otimes\delta_t$,
where $x_t \in M(A)$ for any $t\in G$. Thus
$$
\sum_{s, t\in G}x_s \alpha_t (a)\otimes \delta_s \delta_t =\sum_{s\in G}(a\otimes 1^0 )(x_s \otimes\delta_s ) .
$$
That is,
$$
\sum_{s\in G}x_s \alpha_s (a)\otimes\delta_s =\sum_{s\in G}(a\otimes 1^0 )(x_s \otimes\delta_s) .
$$
Hence
$$
x_s \alpha_s (a)=ax_s
$$
for any $a\in A$, $s\in G$. If $s\ne\imath$, $\alpha_s$ is a free automorphism of $A$. Thus
$x_s =0$ for any $s\in G\setminus\{\imath\}$. Therefore $x=x_{\imath}\otimes\delta_{\imath}$.
Also, $x_{\imath}\alpha_{\imath}(a)=ax_{\imath}$. That is, $x_{\imath}a=ax_{\imath}$. Hence
$x_{\imath}\in A' \cap M(A)$. Since $\delta_{\imath}$ is the distinguished projection in $C(G)$, we
see that $\rho_{\alpha}$ is free.
\newline
$(2)\Rightarrow (1)$: We suppose that $\rho_{\alpha}$ is free. Let $t\in G\setminus\{\imath\}$ and
we suppose that an element
$x\in M(A)$ satisfies that for any $a\in A$, $x\alpha_t (a)=ax$ for any $a\in A$.
For any $s\in G$, let $x_s$ be the element in $M(A)$ defined by
$$
x_s = \begin{cases} x & \text{if $s=t$} \\
1  &  \text{if $s=\imath$} \\
0 & \text{others}
\end{cases} .
$$
Let $y=\sum_{s\in G}x_s \otimes \delta_s$. Then $y\in M(A)\otimes C(G)$ and for any $a\in A$,
$$
y\rho_{\alpha}(a)=\sum_{s, r\in G}(x_s \otimes\delta_s )(\alpha_r (a)\otimes\delta_r )
=\sum_{s\in G}x_s \alpha_s (a)\otimes\delta_s =x\alpha_t (a)\otimes\delta_t +a\otimes\delta_{\imath} .
$$
Also, for any $a\in A$,
$$
(a\otimes 1^0 )y=\sum_{s\in G}ax_s \otimes\delta_s =x\otimes\delta_t +a\otimes\delta_{\imath} .
$$
Hence $y\rho_{\alpha}(a)=(a\otimes 1^0 )y$ for any $a\in A$. Since $\rho_{\alpha}$ is free,
$y\in (A' \cap M(A))\otimes\BC\delta_{\imath}$.
By the definition of $y$, $x=0$. Therefore $\alpha$ is free.
\end{proof}

\begin{remark}\label{remark:coaction4} As Schwieger and Wagner pointed out
in Introduction and Remark of \cite {SW:compact}, the freeness in this paper is different from theirs in
\cite {SW:compact}. In this paper, we use the notion ``saturated" instead of their ``free".
The definition of ``saturated" is given in \cite {SP:saturated}.
\end{remark}
 
\begin{prop}\label{prop:coaction5} Let $\rho$ be a coaction of $H^0$ on a unital $C^*$-algebra $A$. Then
the following are equivalent:
\newline
$(1)$ $\rho$ is free,
\newline
$(2)$ $A' \cap (A\rtimes_{\rho}H)=A' \cap A$,
\newline
$(3)$ The canonical conditional expectation $E_1^{\rho}$ from $A\rtimes_{\alpha}H$ onto $A$ is unique.
\end{prop}
\begin{proof} $(1)\Rightarrow (2)$: Let $x\in A' \cap(A\rtimes_{\rho}H)$. Then we can write that
$$
x=\sum_{i, j, k}x_{ij}^k \rtimes_{\rho}w_{ij}^k ,
$$
where $x_{ij}^k \in A$ for any $i, j, k$. For any $a\in A$,
\begin{align*}
(a\rtimes_{\rho}1)x & =\sum_{i, j, k}ax_{ij}^k \rtimes_{\rho}w_{ij}^k , \\
x(a\rtimes_{\rho}1 ) & =\sum_{i, j, k}(x_{ij}^k \rtimes_{\rho}w_{ij}^k )(a\rtimes_{\rho}1)
=\sum_{i, j. k, j_1}x_{ij}^k [w_{ij_1}^k \cdot_{\rho}a]\rtimes_{\rho}w_{j_1 j}^k \\
& =\sum_{i, j, k, j_1}x_{j_1 j}^k [w_{j_1 i}^k \cdot_{\rho}a]\rtimes_{\rho}w_{ij}^k .
\end{align*}
Since $(a\rtimes_{\rho}1)x=x(a\rtimes_{\rho}1)$,
$$
ax_{ij}^k =\sum_{j_1}x_{j_1 j}^k [w_{j_1 i}^k \cdot_{\rho}a] \quad (*)
$$
for any $i, j, k$ and any $a\in A$. Let $z$ be the element in $A\otimes H^0$ defined by
$$
z=\sum_{i, j, k}x_{ij}^k \otimes\phi_{ji}^k .
$$
We claim that $z\rho(a)=(a\otimes 1^0 )z$ for any $a\in A$. Indeed, for any $r, s, t$,
\begin{align*}
(z\rho(a))^{\widehat{}}(w_{st}^r ) & =\sum_{i, j, k, t_1}x_{ij}^k \phi_{ji}^k (w_{st_1}^r )[w_{t_1 t}^r \cdot_{\rho}a] \\
& =\sum_{t_1}x_{t_1 s}^r [w_{t_1 t}^r \cdot_{\rho}a]=ax_{ts}^r
\end{align*}
by Equation (*). On the other hand, for any $r, s, t$,
$$
[(a\otimes 1^0 )z]^{\widehat{}}(w_{st}^r ) =\sum_{i, j, k}ax_{ij}^k \phi_{ji}^k (w_{st}^r )
=ax_{ts}^r .
$$
Hence $(a\otimes 1^0 )=z\rho(a)$ for any $a\in A$. Since $\rho$ is free, $z\in (A\cap A')\otimes\BC\tau$.
That is,
$$
z=\sum_{i, j, k}x_{ij}^k \otimes\phi_{ji}^k \in (A' \cap A)\otimes\BC\tau .
$$
Since $\tau\in\{\phi_{ij}^k \}$, $z=x_0 \otimes\tau$, where $x_0 \in A' \cap A$. Therefore,
$x=x_0 \rtimes_{\rho}1$ since $\tau(1)=1$.
\newline
$(2)\Rightarrow (1)$: Let $z\in A\otimes H^0$ such that $z\rho(a)=(a\otimes 1^0 )z$ for any $a\in A$.
Then for any $h\in H$,
$$
\widehat{z}(h_{(1)})[h_{(2)}\cdot_{\rho}a]=a\widehat{z}(h) . \quad (**)
$$
Let $x$ be the element in $A\rtimes_{\rho}H$ defined by
$$
x=\sum_{i, j, k}\widehat{z}(w_{ij}^k )\rtimes_{\rho}w_{ji}^k .
$$
Then for any $a\in A$, 
$$
(a\rtimes_{\rho}1)x=\sum_{i, j, k}a\widehat{z}(w_{ij}^k )\rtimes_{\rho}w_{ji}^k .
$$
By Equation (**),
\begin{align*}
x(a\rtimes_{\rho}1) & =\sum_{i, j, k}(\widehat{z}(w_{ij}^k )\rtimes_{\rho}w_{ji}^k )(a\rtimes_{\rho}1 )
=\sum_{i, j, k, i_1}\widehat{z}(w_{ij}^k )[w_{ji_1}^k \cdot_{\rho}a]\rtimes_{\rho}w_{i_1 i}^k \\
& =\sum_{i, k, i_1}a\widehat{z}(w_{ii_1}^k )\rtimes_{\rho}w_{i_1 i}^k .
\end{align*}
Hence $(a\rtimes_{\rho}1)x=x(a\rtimes_{\rho}1)$. That is,
$x\in A' \cap(A\rtimes_{\rho}H)$. By Condition (2), $x\in A' \cap A$. Since
$1\in \{w_{ij}^k \}$, $x=\widehat{z}(1)\rtimes_{\rho}1$ and $\widehat{z}(1)\in A' \cap A$.
Furthermore, $\widehat{z}(w_{ij}^k )=0$ if $w_{ij}^k \ne 1$. Therefore,
$$
z=\widehat{z}(1)\otimes\tau\in (A' \cap A)\otimes\BC\tau
$$
since $\tau\in\{\phi_{ij}^k \}$.
\newline
$(2)\Rightarrow (3)$: Let $F$ be a conditional expectation from $A\rtimes_{\rho}H$ onto $A$.
By Watatani \cite [Proposition 1.4.1]{Watatani:index}, there is an element
$z\in A' \cap(A\rtimes_{\rho}H)$ with $E_1^{\rho} (z)=1$ such that
$$
F(x)=E_1^{\rho} (az)
$$
for any $x\in A\rtimes_{\rho}H$. Then by Condition (2), $z\in A' \cap A$. Thus $z=E_1^{\rho} (z)=1$
and $F(x)=E_1^{\rho} (x)$ for any $x\in A\rtimes_{\rho}H$. Hence $F=E_1^{\rho}$.
\newline
$(3)\Rightarrow (2)$: Since $E_1^{\rho}$ is faithful by \cite [Lemma 3.14]{KT1:inclusion},
we obtain the conclusion by \cite [Lemma 3.1.1]{Zarikian:expectation}.
\end{proof}

Let $\rho$ be a coaction of $H^0$ on a (non-unital) $C^*$-algebra $A$. Then in the same way as in
\cite {Kodaka:equivariance}, we can extend $\rho$ to a coaction $\underline{\rho}$ of $H^0$ on $M(A)$.
We note that $\underline{\rho}$ is strictly continuous.

\begin{lemma}\label{lem:coaction6} With the above notation, $\rho$ is free if and only if $\underline{\rho}$
is free.
\end{lemma}
\begin{proof}
We suppose that $\rho$ is free. Let $x$ be an element in $M(A)\otimes H^0$
such that $x\underline{\rho}(a)=(a\otimes 1^0 )x$ for any $a\in M(A)$. Then $x\rho(a)=(a\otimes 1^0 )x$ for
any $a\in A$ since $\underline{\rho}|_A =\rho$. Since $\rho$ is free, $x\in (A' \cap M(A))\otimes\tau$.
Also, $A' \cap M(A)=M(A)' \cap M(A)$. Indeed, clearly $M(A)' \cap M(A)\subset A' \cap M(A)$.
Let $z\in A' \cap M(A)$. Then $za=az$ for any $a\in A$. Since $A$ is dense in $M(A)$ under the strict
topology. Thus $za=az$ for any $a\in M(A)$. Hence $A' \cap M(A)\subset M(A)' \cap M(A)$. That is,
$A' \cap M(A)=M(A)' \cap M(A)$. Therefore, $x\in (M(A)' \cap M(A))\otimes\tau$. It follows that $\underline{\rho}$
is free. Next, we suppose that $\underline{\rho}$ is free. Let $x$ be an element in $M(A)\otimes H^0$
such that $x\rho(a)=(a\otimes 1^0 )x$ for any $a\in A$. Then since $A$ is dense in $M(A)$ under the strict
topology and $\underline{\rho}$ is strictly continuous, $x\underline{\rho}(a)=(a\otimes 1^0 )x$ for any
$a\in M(A)$. Since $\underline{\rho}$ is free, $x\in (M(A)' \cap M(A))\otimes\tau$. Since
$M(A)' \cap M(A)=A' \cap M(A)$, $x\in (A' \cap M(A))\otimes\tau$. Thus $\rho$ is free.
\end{proof}

Let $\rho$ be a coaction of $H^0$ on a $C^*$-algebra $A$. Let $F^A$ be a conditional expectation from
$A\rtimes_{\rho}H$ onto $A$. We extend $F^A$ to a conditional expectation $F^{M(A)}$ from
$M(A)\rtimes_{\underline{\rho}}H$ onto $M(A)$. We note that
$M(A\rtimes_{\rho}H)=M(A)\rtimes_{\underline{\rho}}H$ by \cite [Lemma 2.10]{Kodaka:equivariance}.
Let $x$ be any element in $M(A)\rtimes_{\underline{\rho}}H$. Then there is a net
$\{x_{\lambda}\}_{\lambda\in\Lambda}\subset  A\rtimes_{\rho}H$ such that $x_{\lambda}$ is strictly convergent
to $x$ in $M(A)\rtimes_{\underline{\rho}}H$. Let $F^{M(A)}$ be the linear map from
$M(A)\rtimes_{\underline{\rho}}H$ onto
$M(A)$ defined by
$$
F^{M(A)}(x)=\lim_{\lambda} F^A (x_{\lambda}) ,
$$
where the limit is taken under the strict topology in $M(A)$. By the easy computations, we see that $F^{M(A)}$ is a
conditional expectation from $M(A)\rtimes_{\underline{\rho}}H$ onto $M(A)$ with
$F^{M(A)}|_{A\rtimes_{\rho}H}=F^A$.

\begin{lemma}\label{lem:coaction7} With the above notation, $F^{M(A)}$ is the unique conditional expectation
from $M(A\rtimes_{\rho}H)$ onto $M(A)$ such that $F^{M(A)}|_{A\rtimes_{\rho}H}=F^A$.
\end{lemma}
\begin{proof} Let $G^{M(A)}$ be a conditional expectation from $M(A\rtimes_{\rho}H)$ onto $M(A)$
such that $G^{M(A)}|_{A\rtimes_{\rho}H}=F^A$. Let $x\in M(A\rtimes_{\rho}H)$ and $a\in A$. Then
$$
||G^{M(A)}(x)a-F^{M(A)}(x)a||=||G^{M(A)}(xa)-F^{M(A)}(xa)||=||F^A (xa)-F^A (xa)||=0 .
$$
Thus $G^{M(A)}(x)=F^{M(A)}(x)$ for any $x\in M(A\rtimes_{\rho}H)$.
\end{proof}

\begin{thm}\label{thm:coaction8} Let $\rho$ be a coaction of $H^0$ on a $C^*$-algebra $A$. Then
the following conditions are equivalent:
\newline
$(1)$ $\rho$ is free,
\newline
$(2)$ $A' \cap M(A\rtimes_{\rho}H)=A' \cap M(A)$,
\newline
$(3)$ The canonical conditional expectation $E_1^{\rho}$ from $A\rtimes_{\rho}H$ onto $A$ is unique.
\end{thm}
\begin{proof} $(1)\Rightarrow (2)$: Since $\rho$ is free, by Lemma \ref {lem:coaction6}, $\underline{\rho}$ is
free. Hence by Proposition \ref {prop:coaction5} and \cite [Lemma 2.10]{Kodaka:equivariance},
$$
M(A)' \cap M(A\rtimes_{\rho}H)=M(A)' \cap (M(A)\rtimes_{\underline{\rho}}H)=M(A)' \cap M(A) .
$$
By the proof of Lemma \ref {lem:coaction6}, $A' \cap M(A)=M(A)' \cap M(A)$. Also, we can see that
$A' \cap M(A\rtimes_{\rho}H)=M(A)' \cap M(A\rtimes_{\rho}H)$ in the same way as in the proof of
Lemma \ref {lem:coaction6}. Thus $A' \cap M(A\rtimes_{\rho}H)=A' \cap M(A)$.
\newline
$(2)\Rightarrow (1)$: By the discussions of $(1)\Rightarrow (2)$,
$M(A)' \cap M(A\rtimes_{\rho}H)=M(A)' \cap M(A)$. Hence by Proposition \ref {prop:coaction5} and
\cite [Lemma 2.10]{Kodaka:equivariance}, $\underline{\rho}$ is free. Thus by Lemma \ref {lem:coaction6},
$\rho$ is free.
\newline
$(2) \Rightarrow (3)$: Let $F^A$ be a conditional expectation from $A\rtimes_{\rho}H$ onto $A$.
By the discussions before this theorem, there is a conditional expectation $F^{M(A)}$ from
$M(A)\rtimes_{\underline{\rho}}H$ onto $M(A)$ extending $F^A$. Also, since
$A' \cap M(A\rtimes_{\rho}H)=A' \cap M(A)$, $M(A)' \cap (M(A)\rtimes_{\underline{\rho}}H)=M(A)' \cap M(A)$.
Thus by Proposition \ref{prop:coaction5}, the canonical conditional expectation $E_1^{\underline{\rho}}$
from $M(A)\rtimes_{\underline{\rho}}H$ onto $M(A)$ is unique. Hence $E_1^{\underline{\rho}}=F^{M(A)}$.
Since $E_1^{\underline{\rho}}|_{A\rtimes_{\rho}H}=E_1^{\rho}$ and
$F^{M(A)}|_{A\rtimes_{\rho}H}=F^A$, $E_1^{\rho} =F^A$.
Therefore we obtain Condition (3).
\newline
$(3)\Rightarrow (2)$: By the discussions of $(2)\Rightarrow (3)$ and Lemma \ref{lem:coaction7},
the canonical conditional $E_1^{\underline{\rho}}$ from $M(A)\rtimes_{\underline{\rho}}H$ onto $M(A)$
is unique. Thus by Proposition \ref{prop:coaction5}, $M(A)' \cap M(A\rtimes_{\rho}H)=M(A)' \cap M(A)$.
Hence $A' \cap M(A\rtimes_{\rho}H)=A' \cap M(A)$.
\end{proof}

\begin{remark}\label{rem:coaction8-2} In the following way, we can see that there is a free coaction of a finite
dimensional $C^*$-Hopf algebra on a simple unital $C^*$-algebra: Let $A\subset B$ be a depth $2$, unital
inclusion of simple unital $C^*$-algebras of Watatani index-finite type. We suppose that $A' \cap B=\BC 1$.
Then by \cite [Corollary 6.4]{Izumi:simple}, there is a coaction $\rho$ of a finite dimensional $C^*$-Hopf
algebra $H^0$ on $B$ such that $A$ is a fixed point $C^*$-algebra $B^{\rho}$. Since $A' \cap B=\BC 1$,
$B' \cap (B\rtimes_{\rho}H)=\BC 1$ by the proof of \cite [Proposition 2.7.3]{Watatani:index}. Hence
by Proposition \ref{prop:coaction5}, $\rho$ is a free coaction of $H^0$ on $B$.
\end{remark}

Following Blattner, Cohen and Montgomery \cite {BCM:crossed}, we give the definitions of an inner coaction
and an outer coaction. Let $\rho$ be a coaction of $H^0$ on a $C^*$-algebra $A$. Let $\rho_{H^0}^A$ be
the trivial coaction of $H^0$ on $A$.
\begin{Def}\label{coaction:9} (1) $\rho$ is
\sl
inner
\rm
if there is a unitary element $u$ in $M(A)\otimes H^0$
satisfying the following:
\newline
(i) $\rho=\Ad(u)\circ\rho_{H^0}^A$,
\newline
(ii) $(u\otimes 1^0 )(\underline{\rho_{H^0}^A}\otimes \id_{H^0})(u)=(\id_{M(A)}\otimes\Delta^0 )(u)$,
\newline
where $\underline{\rho_{H^0}^A}$ is the coaction of $H^0$ on $M(A)$ induced by $\rho_{H0}^A$,
that is, $\underline{\rho_{H^0}^A}=\rho_{H^0}^{M(A)}$, the trivial coaction of $H^0$ on $M(A)$.
\newline
(2) $\rho$ is
\sl
outer
\rm
if $H^0$ is not trivial and if, whenever $\pi^0: H^0 \rightarrow K^0$ is a surjective $C^*$-Hopf algebra
homomorphism such that the induced $K^0$-coaction $\sigma=(\id_A \otimes\pi^0 ) \circ\rho$ is inner,
then $K^0$ is trivial.
\end{Def}

\begin{prop}\label{prop:coaction10} Let $\rho$ be a coaction of a non-trivial $C^*$-Hopf algebra
$H^0$ on a $C^*$-algebra $A$. If $\rho$ is free, then $\rho$ is outer.
\end{prop}
\begin{proof}Let $\pi^0$ be a surjective $C^*$-Hopf algebra homomorphism of $H^0$ onto $K^0$
and let $\sigma$ be the coaction of $K^0$ on $A$ induced by $\pi^0$, that is,
$\sigma=(\id_A \otimes\pi^0 )\circ\rho$. Let $K$ be the $C^*$-Hopf algebra induced by $K^0$ and let
$\pi$ be the $C^*$-Hopf algebra homomorphism induced by $\pi^0$, which is defined by
$$
\phi(\pi(k))=\pi^0 (\phi)(k)
$$
for any $\phi\in H^0$, $k\in K$. Then $\pi$ is injective since $\pi^0$ is surjective.
Also, since $\sigma=(\id_A \otimes\pi^0 )\circ\rho$, $A\rtimes_{\sigma}K\subset A\rtimes_{\rho}H$.
Furthermore, since $\overline{(A\rtimes_{\sigma}K)(A\rtimes_{\rho}H)}=A\rtimes_{\rho}H$,
$M(A\rtimes_{\sigma}K)\subset M(A\rtimes_{\rho}H)$. Since $\rho$ is free, by Theorem \ref{thm:coaction8},
$A' \cap M(A\rtimes_{\rho}H)=A' \cap M(A)$. Since
$A' \cap M(A)\subset A' \cap M(A\rtimes_{\sigma}K)\subset A' \cap M(A\rtimes_{\rho}H)$,
$$
A' \cap M(A\rtimes_{\sigma}H)=A' \cap M(A) .
$$
Thus $\sigma$ is free by Theorem \ref{thm:coaction8}. We suppose that $\sigma$ is inner. Then there is
a unitary element $u\in M(A\otimes K^0 )$ such that
$$
\sigma=\Ad(u)\circ\rho_{K^0}^A , \quad (u\otimes 1^0)(\underline{\rho_{K^0}^A} \otimes\id_{K^0})(u)
=(\id_{M(A)}\otimes\Delta_{K^0}^0 )(u) ,
$$
where $\Delta_{K^0}^0$ is the comultiplication of $K^0$.
Since $\sigma$ is free, $u\in (A' \cap M(A))\otimes\BC\tau_{K^0}$, where $\tau_{K^0}$ is the distinguished
projection in $K^0$.
Since $u$ is a unitary element, $\tau_{K^0}=1_{K^0}$. Hence
$K^0=\BC 1_{K^0}$, where 1$_{K^0}$ is the unit element
in $K^0$. Therefore, $K^0$ is trivial. Hence
$\rho$ is outer.
\end{proof}

We note that the Rokhlin property implies the outerness in the case that $A$ is a unital $C^*$-algebra.
Before we show it, we give the definitions of the approximate representability and the Rokhlin property.
For a unital $C^*$-algebra $A$, we set
$$
c_0 (A)=\{(a_n )\in l^{\infty}(\BN, A) \, | \, \lim_{n\to\infty}||a_n ||=0 \} , \quad
A^{\infty}=l^{\infty}(\BN, A)/ c_0 (A) .
$$
We denote an element in $A^{\infty}$ by the same symbol $(a_n )$ in $l^{\infty}(\BN, A)$.
We identify $A$ with the $C^*$-subalgebra of $A^{\infty}$ consisting of the equivalence
classes of constant sequences and set
$$
A_{\infty} =A^{\infty} \cap A' .
$$
For a coaction $\rho$ of $H^0$ on $A$, let $\rho^{\infty}$ be the coaction of $H^0$ on $A^{\infty}$
defined by
$$
\rho^{\infty}((a_n ))=(\rho(a_n ))
$$
for any $(a_n )\in A^{\infty}$.

\begin{Def}\label{def:coaction11} Let $\rho$ be a coaction of $H^0$ on a unital $C^*$-algebra $A$.
We say that $\rho$ is
\sl
approximately representable
\rm
if there is a unitary element $w\in A^{\infty}\otimes H^0$ satisfying the following conditions:
\newline
(1) $\rho(a)(\Ad(w)\circ\rho_{H^0}^A )(a)$ for any $a\in A$,
\newline
(2) $(w\otimes 1^0 )(\rho_{H^0}^{A^{\infty}}\otimes\id)(w)=(\id\otimes\Delta^0 )(w)$,
\newline
(3) $(\rho^{\infty}\otimes\id)(w)(w\otimes 1^0 )=(\id\otimes\Delta^0 )(w)$.
\end{Def}

\begin{Def}\label{def:coaction12} Let $\rho$ be a coaction of $H^0$ on a unital $C^*$-algebra
$A$. We say that $\rho$ has
\sl
the Rokhlin property
\rm
if the dual coaction $\widehat{\rho}$ of $H$ on $A\rtimes_{\rho}H$ is approximately representable.
\end{Def}

\begin{prop}\label{prop:coaction13}Let $\rho$ be a coaction of a non-trivial finite dimensional
$C^*$-Hopf algebra $H^0$ on a unital $C^*$-algebra $A$. If $\rho$ has the Rokhlin property, then
$\rho$ is outer.
\end{prop}
\begin{proof} Let $\pi^0$, $\pi$ and $K^0$, $K$ be as in the proof of Proposition \ref{prop:coaction10}.
We regard $K$ as a $C^*$-Hopf subalgebra of $H$ by $\pi$. Also, let $\sigma=(\id\otimes\pi^0 )\circ\rho$,
a coaction of $K^0$ on $A$. Since $\rho$ has the Rokhin property, by \cite [Corollary 6.4]{KT2:coaction},
there is a projection $p\in A_{\infty}$ such that $e\cdot_{\rho^{\infty}}p=\frac{1}{N}$, where
$N=\dim H$. Let $f$ be the distinguished projection in $K$. Then
$$
f\cdot_{\sigma^{\infty}}p=(\id\otimes f)(\sigma^{\infty}(p))=(\id\otimes(f\circ\pi^0 ))(\rho^{\infty}(p))
=(\id\otimes e)(\rho^{\infty}(p))=e\cdot_{\rho^{\infty}}p=\frac{1}{N}
$$
since $\pi(f)=e$. We suppose that $\sigma$ is inner. Then there is a unitary element $u\in A\otimes K^0$
such that
$$\sigma=\Ad(u)\circ\rho_{K^0}^A , \quad
(u\otimes 1^0 )(\rho_{K^0}^A \otimes \id)(u)=(\id\otimes\Delta_{K^0}^0 )(u) ,
$$
where $\Delta_{K^0}^0$ is the comultiplication of $K^0$. Hence
$$
f\cdot_{\sigma^{\infty}}p=\widehat{u}(f_{(1)})[f_{(2)}\cdot_{\rho_{K^0}^A} p]\widehat{u^* }(f_{(3)})
=\widehat{u}(f_{(1)})p\widehat{u^* }(f_{(2)})=p
$$
since $p\in A_{\infty}$. Hence $\frac{1}{N}$ is a projection since $\frac{1}{N}=e$. Thus $N=1$.
This is a contradiction. Therefore, $\rho$ is outer.
\end{proof}

\section{Strong Morita equivalence}\label{sec:SM} Let $A\subset C$ and $B\subset D$ be unital
inclusions of unital $C^*$-algebras. We suppose that $A\subset C$ and $B\subset D$ are strongly
Morita equivalent with respect to a $C-D$-equivalence bimodule $Y$ and its closed subspace $X$.
The definition of strong Morita equivalence for inclusions of $C^*$-algebras is given in \cite {KT4:morita}.
Let $\CalE(A, C)$ be the set of all conditional expectations from $C$ onto $A$. Also, we define
$\CalE(B, D)$ as above. In this section,
we show that strong Morita equivalence for coactions
preserves the freeness of coactions of a finite dimensional $C^*$-Hopf algebra on unital $C^*$-algebras. 
Also, in the same way as above, we show that strong Morita equivalence for twisted
actions preserves the freeness of twisted actions of a countable discrete group on unital $C^*$-algebras.
By \cite [Propositions 2.4 and 5.2]{Kodaka:conditional}, we obtain the following proposition:

\begin{prop}\label{prop:SM5} Let $A\subset C$ and $B\subset D$ be unital inclusions of unital $C^*$-algebras.
We suppose that $A\subset C$ and $B\subset D$ are strongly Morita equivalent with respect to
a $C-D$-equivalence bimodule $Y$ and its closed subspace $X$. Then there is a bijective map from
$\CalE(A, C)$ onto $\CalE(B, D)$.
\end{prop}

By the above proposition, we obtain the following:

\begin{thm}\label{thm:SM6} $(1)$ Let $\rho$ and $\sigma$ be coactions of a finite dimensional $C^*$-Hopf
algebra $H^0$ on unital $C^*$-algebras $A$ and $B$, respectively. We suppose that $\rho$ and $\sigma$ are
strongly Morita equivalent. Then $\rho$ is free if and only if $\sigma$ is free.
\newline
$(2)$ Let $(\alpha, u)$ and $(\beta, v)$ be twisted actions of a countable discrete group $G$ on unital
$C^*$-algebras $A$ and $B$, respectively. We suppose that $(\alpha, u)$ and $(\beta, v)$ are strongly
Morita equivalent. Then $(\alpha, u)$ is free if and only if $(\beta, v)$ is free.
\end{thm}
\begin{proof} (1) We note that the unital inclusions $A\subset A\rtimes_{\rho}H$ and
$B\subset B\rtimes_{\sigma}H$ are strongly Morita equivalent by \cite [Examples]{KT4:morita}
since $\rho$ and $\sigma$ are strongly Morita equivalent. Since $\rho$ is free, $E_1^{\rho}$ is the 
unique conditional expectation from $A\rtimes_{\rho}H$ onto $A$ by Proposition \ref{prop:coaction5}.
Hence by Proposition \ref{prop:SM5}, $E_1^{\sigma}$ is the unique conditional expectation from
$B\rtimes_{\sigma}H$ onto $B$. Thus $\sigma$ is free by Proposition \ref{prop:coaction5}.
\newline
(2) Let $E_1^{\alpha, u}$ and $E_1^{\beta, v}$ be the canonical conditional expectations from
$A\rtimes_{\alpha, u, r}G$ and $B\rtimes_{\beta, v, r}G$ onto $A$ and $B$, respectively, where
$A\rtimes_{\alpha, u, r}G$ and $B\rtimes_{\beta, v, r}G$ are the reduced twisted crossed products of
$A$ and $B$ by $(\alpha, u)$ and $(\beta, v)$, respectively. Then by \cite [Proposition 2.1]{Kodaka:countable},
the inclusions $A\rtimes_{\alpha, u, r}G$ and $B\rtimes_{\beta, v, r}G$ are strongly Morita equivalent.
Since $(\alpha, u)$ is free, $E_1^{\alpha, u}$ is the unique conditional expectation from
$A\rtimes_{\alpha, u, r}G$ onto $A$ by \cite [Proposition 4.1]{Kodaka:countable}. Hence by Proposition
\ref{prop:SM5}, $E_1^{\beta, v}$ is the unique conditional expectation from $B\rtimes_{\beta, v, r}G$ onto
$B$. Thus $(\beta, v)$ is free by \cite [Proposition 4.1]{Kodaka:countable}.
\end{proof}

Theorem \ref{thm:SM6} shows that strong Morita equivalence for coactions of a finite dimensional
$C^*$-Hopf algebra on unital $C^*$-algebras preserves the freeness of coactions. But strong
Morita equivalence for coactions dos not preserve the saturatedness. We give such an example.
Let $p$ and $q$ be projections in a unital $C^*$-algebra $A$. If $p$ is Murray-von Neumann equivalent
to $q$ in $A$, we denote it by $p\sim q$ in $A$.

\begin{exam}\label{exam:SM7}
\rm
We give coactions $\rho$ and $\sigma$ of $H^0$ on a unital $C^*$-algebras
$A$ and $B$, respectively, which are strongly Morita equivalent and have the following properties,
where $H^0$ is the dual $C^*$-Hopf algebra of a finite dimensional $C^*$-Hopf algebra $H$.
\newline
(1) $\rho$ is not saturated,
\newline
(2) $\sigma$ is saturated,
\newline
(3) $\rho$ and $\sigma$ are strongly Morita equivalent.
\newline
Let $\rho$ be a coaction of $H^0$ on $A$, which is not saturated. Let $\sigma$ be the second dual coaction
$\rho$ of $H^0$ on $A\rtimes_{\rho}\rtimes_{\widehat{\rho}}H^0$. Let
$B=A\rtimes_{\rho}\rtimes_{\widehat{\rho}}H^0$. Then $\sigma$ is a coaction of $H^0$ on $B$ and by
\cite [Proposition 3.19]{KT1:inclusion},
$\widehat{\sigma}(1\rtimes_{\sigma}e)\sim(1\rtimes_{\sigma}e)\otimes 1$ in $(B\rtimes_{\sigma}H)\otimes H$.
Hence by \cite [Proposition 6.4]{KT1:inclusion}, $\sigma$ is saturated. Also, by \cite [Theorem 3.3]
{KT2:coaction}, there is an isomorphism $\Psi$ of $M_N (A)$ onto $B$ such that $\sigma$ is exterior
equivalent to $(\Psi\otimes\id)\circ(\rho\otimes\id)\circ\Psi^{-1}$. Hence by \cite [Lemmas 3.11 and 3.16]
{KT3:equivalence}, $\rho$ and $\sigma$ are strongly Morita equivalent.
\end{exam}

\section{Simplicity and freeness}\label{sec:SF} In this section, we discuss simplicity of $C^*$-algebras and
freeness of coactions of finite dimensional $C^*$-Hopf algebras.

\begin{lemma}\label{lem:SF1} Let $A$ be a simple $C^*$-algebra and let $\rho$ be a free coaction
of a finite dimensional $C^*$-Hopf algebra $H^0$ on $A$. Then $A\rtimes_{\rho}H$ is simple.
\end{lemma}
\begin{proof} We assume that $A\rtimes_{\rho}H$ is not simple. Then by Izumi \cite [Theorem 3.3]{Izumi:simple},
$A\rtimes_{\rho}H$ is a finite direct sum of simple closed two-sided ideals, that is,
$A\rtimes_{\rho}H=\oplus_{i=1}^n I_i$, where $I_i$ is a simple closed two-sided ideal of $A\rtimes_{\rho}H$
for $i=1,2,\dots, n$ and $n\geq 2$. Thus,
\begin{align*}
(A\rtimes_{\rho}H)' \cap M(A\rtimes_{\rho}H) & =(\oplus_{i=1}^n I_i )' \cap M(\oplus_{i=1}^n I_i )
\cong (\oplus_{i=1}^n I_i )' \cap(\oplus_{i=1}^n M(I_i )) \\
& \supset \oplus_{i=1}^n (I_i ' \cap M(I_i ))
\cong \underbrace{\BC1\oplus \cdots \oplus \BC1}_{n-\text{times}}
\end{align*}
by Pedersen \cite [Corollary 4.4.8]{Pedersen:auto}. On the other hand, since $\rho$ is free,
by Theorem \ref{thm:coaction8} and Pedersen \cite [Corollary 4.4.8]{Pedersen:auto},
$$
\BC1=A' \cap M(A)=A' \cap M(A\rtimes_{\rho}H)\supset (A\rtimes_{\rho}H)' \cap M(A\rtimes_{\rho}H) .
$$
This is a contradiction since $n\geq 2$. Therefore, we obtain the conclusion.
\end{proof}

\begin{lemma}\label{lem:SF2} Let $A$ be a simple unital $C^*$-algebra and let $\rho$ be a free coaction
of $H^0$ on $A$. Let $\widehat{\rho}$ be its dual coaction of $H$ on $A\rtimes_{\rho}H$. Then $\widehat{\rho}$ is
free.
\end{lemma}
\begin{proof} Since $\rho$ is free, by Lemma \ref{lem:SF1}, $A\rtimes_{\rho}H$ is simple
unital $C^*$-algebra. Hence $(A\rtimes_{\rho}H)' \cap (A\rtimes_{\rho}H)=\BC 1$
by Pedersen \cite [Corollary 4.4.8]{Pedersen:auto}. Also, by Watatani \cite [Lemma 2.10.6]{Watatani:index}
and the proof of \cite [Proposition 2.7.3]{Watatani:index},
$$
\dim (A' \cap (A\rtimes_{\rho}H))=\dim ((A\rtimes_{\rho}H)' \cap(A\rtimes_{\rho}\rtimes_{\widehat{\rho}}H^0 )) .
$$
Since $A' \cap (A\rtimes_{\rho}H)=A' \cap A=\BC1$ by Theorem \ref{thm:coaction8}, we can see that
$$
(A\rtimes_{\rho}H)' \cap(A\rtimes_{\rho}\rtimes_{\widehat{\rho}}H^0 )=\BC1 .
$$
Thus by Theorem \ref{thm:coaction8}, $\widehat{\rho}$ is free.
\end{proof}

\begin{prop}\label{prop:SF3} Let $A$ be a simple unital $C^*$-algebra. Let $\rho$ be a coaction
of $H^0$ on $A$ and $\widehat{\rho}$ its dual coaction of $H$ on $A\rtimes_{\rho}H$. Then
$\rho$ is free if and only if $\widehat{\rho}$ is free.
\end{prop}
\begin{proof} We suppose that $\rho$ is free. Then by Lemma \ref{lem:SF2}, $\widehat{\rho}$ is free.
We suppose that $\widehat{\rho}$ is free. Then since $A\rtimes_{\rho}H$ is a simple unital $C^*$-algebra
by Lemma \ref{lem:SF1}, $\widehat{\widehat{\rho}}$ is free by Lemma \ref{lem:SF2}. By
\cite [Theorem 3.3]{KT2:coaction}, $\widehat{\widehat{\rho}}$ and $\rho$ are strongly Morita equivalent.
Hence by Theorem \ref{thm:SM6}, $\rho$ is a free coaction of $H^0$ on $A$.
\end{proof}

\end{document}